\documentclass[12pt,a4paper]{article}

\usepackage[dvips]{color}
\usepackage{color}
\usepackage{xcolor}

\usepackage{amsfonts}
\usepackage{amsfonts}
\usepackage{amsfonts}
\usepackage{amsfonts}
\usepackage{amsfonts}
\usepackage{mathrsfs}
\usepackage{amsfonts}
\usepackage{mathrsfs}
\usepackage{amsfonts}
\usepackage{amsfonts}
\usepackage{mathrsfs}
\usepackage{mathrsfs}
\usepackage{amsfonts}
\usepackage{amsfonts}
\usepackage{amsfonts}
\usepackage{amsfonts}
\usepackage{mathrsfs}

\usepackage{amsfonts}
\usepackage{mathrsfs}
\usepackage{mathrsfs}
\usepackage{mathrsfs}
\usepackage{mathrsfs}
\usepackage{amsfonts}
\usepackage{mathrsfs}
\usepackage{mathrsfs}
\usepackage{amsfonts}
\usepackage{amsfonts}
\usepackage{amsfonts}
\usepackage{amsfonts}
\usepackage{amsfonts}
\usepackage{amsfonts}
\usepackage{amsfonts}
\usepackage{amsfonts}
\usepackage{amsfonts}
\usepackage{mathrsfs}
\usepackage{amsfonts}
\usepackage{mathrsfs}
\usepackage{amsfonts}
\usepackage{mathrsfs}
\usepackage{amsfonts}
\usepackage{amsfonts}
\usepackage{amsfonts}
\usepackage{amsfonts}
\usepackage{amsfonts}
\usepackage{amsfonts}
\usepackage{amsfonts}
\usepackage{amsfonts}
\usepackage{amsfonts}
\usepackage{amsfonts}
\usepackage{amsfonts}
\usepackage{amsfonts}
\usepackage{mathrsfs}
\usepackage{mathrsfs}
\usepackage{mathrsfs}
\usepackage{mathrsfs}
\usepackage{mathrsfs}
\usepackage{mathrsfs}
\usepackage{mathrsfs}
\usepackage{mathrsfs}
\usepackage{mathrsfs}
\usepackage{amsfonts}
\usepackage{amsfonts}
\usepackage{amsfonts}
\usepackage{amsfonts}
\usepackage{amsfonts}
\usepackage{amsfonts}
\usepackage{mathrsfs}
\usepackage{amsfonts}
\usepackage{amsfonts}
\usepackage{amsfonts}
\usepackage{amsfonts}
\usepackage{amsfonts}
\usepackage{amsfonts}
\usepackage{amsfonts}
\usepackage{amsfonts}
\usepackage{amsfonts}
\usepackage{amsfonts}
\usepackage{amsfonts}
\usepackage{amsfonts}
\usepackage{mathrsfs}
\usepackage{amsfonts}
\usepackage{mathrsfs}
\usepackage{amsfonts}
\usepackage{amsfonts}
\usepackage{amsfonts}
\usepackage{mathrsfs}
\usepackage{amsfonts}
\usepackage{amsfonts}
\usepackage{amsfonts}
\usepackage{amsfonts}
\usepackage{amsmath}
\usepackage{mathrsfs}
\usepackage{amsfonts}
\usepackage{amsfonts}
\usepackage{amsfonts}
\usepackage{amsfonts}
\usepackage{amsfonts}
\usepackage{amsfonts}
\usepackage{amsfonts}
\usepackage{amsfonts}
\usepackage{amsfonts}
\usepackage{amsfonts}
\usepackage{amsfonts}
\usepackage{mathrsfs}
\usepackage{amsfonts}
\usepackage{amsfonts}
\usepackage{amsfonts}
\usepackage{amsfonts}
\usepackage{amsmath}
\usepackage{mathrsfs}
\usepackage{mathrsfs}
\usepackage{mathrsfs}
\usepackage{mathrsfs}
\usepackage{mathrsfs}
\usepackage{mathrsfs}
\usepackage{mathrsfs}
\usepackage{mathrsfs}
\usepackage{amsfonts}
\usepackage{amsfonts}
\usepackage{amsfonts}
\usepackage{amsfonts}
\usepackage{amsfonts}
\usepackage{amsfonts}
\usepackage{amsfonts}
\usepackage{amsfonts}
\usepackage{amsfonts}
\usepackage{amsfonts}
\usepackage{amsfonts}
\usepackage{amsfonts}
\usepackage{mathrsfs}
\usepackage{amsfonts}
\usepackage{mathrsfs}
\usepackage{amsfonts}
\usepackage{amsfonts}
\usepackage{amsfonts}
\usepackage{amsfonts}
\usepackage{amsfonts}
\usepackage{amsfonts}
\usepackage{amsfonts}
\usepackage{amsfonts}
\usepackage{amsfonts}
\usepackage{amsfonts}
\usepackage{amsfonts}
\usepackage{amsfonts}
\usepackage{amsfonts}
\usepackage{amsfonts}
\usepackage{amsfonts}
\usepackage{amsfonts}
\usepackage{amsfonts}
\usepackage{amsfonts}
\usepackage{mathrsfs}
\usepackage{mathrsfs}
\usepackage{mathrsfs}
\usepackage{mathrsfs}
\usepackage{amsfonts}
\usepackage{amsfonts}
\usepackage{amsfonts}
\usepackage{amsfonts}
\usepackage{amsfonts}
\usepackage{amsfonts}
\usepackage{amsfonts}
\usepackage{amsfonts}
\usepackage{amsfonts}
\usepackage{amsfonts}
\usepackage{amsfonts}
\usepackage{amsfonts}
\usepackage{amsfonts}
\usepackage{amsfonts}
\usepackage{amsfonts}
\usepackage{amsfonts}
\usepackage{amsfonts}
\usepackage{mathrsfs}
\usepackage{mathrsfs}
\usepackage{amsfonts}
\usepackage{amsfonts}
\usepackage{amsfonts}
\usepackage{amsfonts}
\usepackage{amsfonts}
\usepackage{amsfonts}
\usepackage{amsfonts}
\usepackage{amsfonts}
\usepackage{amsfonts}
\usepackage{amsfonts}
\usepackage{amsfonts}
\usepackage{amsfonts}
\usepackage{amsfonts}
\usepackage{mathrsfs}
\usepackage{mathrsfs}
\usepackage{amsfonts}
\usepackage{amsfonts}
\usepackage{amsfonts}
\usepackage{amsfonts}
\usepackage{amsfonts}
\usepackage{amsfonts}
\usepackage{mathrsfs}
\usepackage{mathrsfs}
\usepackage{amsfonts}
\usepackage{amsfonts}
\usepackage{amsfonts}
\usepackage{amsfonts}
\usepackage{amsfonts}
\usepackage{amsfonts}
\usepackage{amsfonts}
\usepackage{amsfonts}
\usepackage{amsfonts}
\usepackage{amsfonts}
\usepackage{amsfonts}
\usepackage{amsfonts}
\usepackage{amsfonts}
\usepackage{amsfonts}
\usepackage{mathrsfs}
\usepackage{amsfonts}
\usepackage{amsfonts}
\usepackage{amsfonts}
\usepackage{amsfonts}
\usepackage{amsfonts}
\usepackage{amsfonts}
\usepackage{amsfonts}
\usepackage{amsfonts}
\usepackage{amsfonts}
\usepackage{amsfonts}
\usepackage{amsfonts}
\usepackage{amsfonts}
\usepackage{amsfonts}
\usepackage{amsfonts}
\usepackage{amsfonts}
\usepackage{amsfonts}
\usepackage{amsmath}
\usepackage{mathrsfs}
\usepackage{amsfonts}
\usepackage{amsfonts}
\usepackage{amsfonts}
\usepackage{amsfonts}
\usepackage{amsmath}
\usepackage{mathrsfs}
\usepackage{amsfonts}
\usepackage{mathrsfs}
\usepackage{mathrsfs}
\usepackage{mathrsfs}
\usepackage{mathrsfs}
\usepackage{amsmath}
\usepackage{mathrsfs}
\usepackage{amsfonts}
\usepackage{color}
\usepackage{mathrsfs}

\usepackage{stmaryrd}
\usepackage{amsmath}
\usepackage{amssymb}

\usepackage{bbm}
\usepackage{mathrsfs}
\usepackage{graphicx}
\usepackage{cite}

\usepackage{enumerate}
\usepackage{thm}
\makeatletter
\def\tank#1{\protected@xdef\@thanks{\@thanks
 \protect\footnotetext[0]{#1}}}
\def\bigfoot{

 \@footnotetext}
\makeatother

\newcommand{\ea}{\end{array}}

\allowdisplaybreaks

\newtheorem{theorem}{Theorem}[section]
\newtheorem{prp}[theorem]{Proposition}
\newtheorem{lemma}[theorem]{Lemma}
\newtheorem{Assumption}{Assumption}[section]
\theoremstyle{definition}
\newtheorem{definition}[theorem]{Definition}

\theoremstyle{remark}
\newtheorem{remark}[theorem]{Remark}

\numberwithin{equation}{section}

{\theorembodyfont{\rmfamily}
}

\title{Large deviation principles and Malliavin derivative for mean reflected stochastic differential equations}

\footnotesize{\author{
Ping Chen$^{1,}$\thanks{chenping@mail.ustc.edu.cn},\ \
 Jianliang Zhai$^{1,}$\thanks{zhaijl@ustc.edu.cn}\\
 {\em $^1$ School of Mathematical Sciences,}\\
 {\em University of Science and Technology of China,}\\
 {\em Hefei, 230026, China}%\\
}
\date{}
\newenvironment{proof}{\par\noindent{\bf Proof:}}{\hspace*{\fill}$\blacksquare$\par}
\begin{document}
\maketitle
%\begin{minipage}{140mm}
%\begin{center}
\noindent \textbf{Abstract:}
In this paper, we consider a class of reflected stochastic differential equations for which the constraint is not on the paths of the solution but on its law. We establish a small noise large deviation principle, a large deviation for short time and the Malliavin derivative. To prove large deviation principles, a sufficient
condition for the weak convergence method, which is suitable for Mckean-Vlasov stochastic
differential equation, plays an important role.
%\end{center}

%\end{minipage}

\vspace{4mm}

%\noindent \textbf{AMS Subject Classification}: Primary 60H15 Secondary 35R60, 37L55.

\vspace{3mm}
\numberwithin{equation}{section}

\section{Introduction}
In this paper, we study a class of mean reflected stochastic differential equation (SDE):
\begin{align}
\left\{
\begin{aligned}\label{II-eq1.1}
& X_{t}=\xi + \int_0^t b(X_{s})\mathrm{d}s+\ \int_0^t \sigma(X_{s})\mathrm{d}B_{s}+\ K_{t},\ t\in [0,T],  \\
& \mathbb{E}[h(X_{t})]\geq 0, \quad
\int_0^t \mathbb{E}[h(X_{s})]\mathrm{d}K_{s}=0,\ t\in[0,T],\\
\end{aligned}
\right.
\end{align}
where the initial date $\xi \in \mathbb{R}$, $\{ B_t \}_{t \geq 0}$ is a  one-dimensional standard Brownian motion defined on a complete probability space
$(\Omega,\mathcal{F},\mathbb{F},\mathbb{P})$ with the augmented filtration $\mathbb{F}:=\{{\mathcal{F}}_t\}_{t\geq 0}$ generated by $\{ B_t \}_{t \geq 0}$, $b$, $\sigma$ and $h$ are given functions from $\mathbb{R}$ to $\mathbb{R}$. For the precise conditions on $b$, $\sigma$ and $h$, we refer the reader to Section 2.

%Let $(\tilde \Omega, \tilde {\mathcal{F}}, \tilde{\mathbb{F}}, \tilde {\mathbb{P}})$ is a filtered probability space with
%$\tilde{\mathbb{F}}:= \{ \tilde{\mathcal{F}_t} \}_{t \in [0.1]}$ the augmented filtration generated by a one-dimensional standard Brownian  motion $\{\tilde{B_t}\}_{t \in [0,T]}$.

Reflected SDEs have been widely studied in stochastic analysis since the works by Skorokhod \cite{SAV1,SAV2}. There is an extensive literature on reflected differential equations with different kinds of reflections, including normal reflection \cite{T,LS,S}, oblique reflection \cite{C,DI,WYZ2,WeiYangZhai} and sticky reflection \cite{GV,GM,G1,G2,EP}, etc. In this current work, we focus on mean reflected SDEs in which the constraint is on the law of the solution rather than on its paths. Motivated by super-hedging of claims under running risk management constraints, the first paper to study mean reflected SDEs was \cite{BEH}, which considered the backward form and proved the existence and uniqueness. In \cite{BDG}, the authors studied the forward form, and its approximation by a suitable interacting particle system and numerical schemes. In \cite{DL}, the authors got the Talagrand quadratic transportation cost inequality for the law of the solution of mean reflected backward SDE. We also refer to \cite{BGL} for the case of jumps and \cite{BCPD} for the mulit-dimensional case.

Large deviation principle (LDP) plays an important role in stochastic analysis, which can provide an exponential estimate for the probability of the rare event in terms of some explicit rate function. The earliest framework was introduced by Varadhan in \cite{Varadhan1966Asymptotic} and \cite{varadhan1967diffusion}, in which the small noise and small time LDP for finite dimensional diffusion processes
were studied, respectively. The literature on the problem is huge, and  a listing of them can be found in \cite{WYJT}.
Small noise (also called Freidlin--Wentzell type) LDP for SDEs with normal reflection or oblique reflection have been explored in a number of works; see, e.g. \cite{RXZ,MSW,RWZZZ,AdamsReisRavaille} and the references therein.  To the best of our knowledge, there is no results on small time LDP for SDEs with reflection.

%To the best of our knowledge, there is only one paper considering the large deviation principle for the mean reflected stochastic differential equations.
%In \cite{LY}, the author used the weak convergence method to establish the Freidlin--Wentzell type large deviation principle for the mean reflected stochastic differential equations with jumps. However, unfortunately, the rate function obtained in \cite{LY} is wrong.

In this paper, our first aim is to establish the small noise and small time LDPs for (\ref{II-eq1.1}).
Compared with the previous corresponding results, new difficulties occur due to the fact that the mean reflection process $K$ depends on the law of the position.

 The first paper dealing with  LDP for reflected Mckean-Vlasov SDE was published by \cite{AdamsReisRavaille}, in which  LDP for reflected Mckean-Vlasov SDE with normal reflection on time independent convex domain is established. The authors in \cite{AdamsReisRavaille} adopted the exponential equivalence arguments, certain time discretization and approximating technique, assuming that the coefficients satisfy some
extra time $\rm H\ddot{o}lder$ continuity conditions; see Assumption 4.1 in \cite{AdamsReisRavaille}. These approaches are very difficult to be applied to the case of infinite dimensional situations and the case of L\'evy driving noise, and requires stronger conditions on the coefficients as mentioned above.

The weak convergence method is proved to be a powerful tool to establish LDPs for various dynamical systems. Recently, the second author of this paper and his collaborators in \cite{WYJT} presented a sufficient condition to prove the criteria of the weak convergence method.  The sufficient condition is suitable for proving LDPs for distribution-dependent SDEs in finite and infinite dimensions.
The second author of this paper and his collaborators in \cite{WeiYangZhai} have used this sufficient condition to prove a Fredlin-Wentzell type  LDP for a class of Mckean-Vlasov SDE with oblique reflection over an non-smooth time dependent domain. Their proof seems to be smoother than that in \cite{AdamsReisRavaille}, and no extra regularity with respect to time on the coefficients is required.
 To study LDPs for (\ref{II-eq1.1}), we will
use the sufficient condition introduced in \cite{WYJT} and a representation formula of $K$. Let us stress that the
key to fully use the weak convergence method to establish LDPs for distribution-dependent SDEs
is to find the correct stochastic control equations; see Example 1.1 in \cite{WYJT} and (\ref{III-eq2.18}) and Remark \ref{rem 1} in this paper.

The second aim of this paper is to study the Malliavin derivative for (\ref{II-eq1.1}), and as an application of this result, if $\sigma$ is nondegenerate, i.e., $\sigma(x)\neq0$ for all $x\in\mathbb{R}$,  then for any $0 \leq t \leq T$, the law of $X_t$ is absolutely continuous with respect to the Lebesgue measure on $\mathbb{R}$.

%The main contribution of this paper is to find the correct rate function.

%Compared with the proofs of the above mentioned corresponding results for normal reflection and oblique reflection, new difficulties occur due to the fact that the mean reflection process $K$ depends on the law of the position.

The organization of the paper is as follows. In section 2, we introduce the basic concepts, notations and assumptions. The small noise LDP for equation \eqref{II-eq1.1} is established in section 3. Then we give the large deviations for short time in section 4. Finally, in section 5, we study the Malliavin derivative for \eqref{II-eq1.1} and its application.

%%%%%%%%%%%%%%%%%%%%%%%%%%%%%%%%%%%%%%%%%%%%%%%%%%%%%%%%%%%%%%%%%%%%%%%%%%%%%%%%%%%%%%%%%%%%%%%%%%%%%%%%%%%%%%%%%%%%%%%%%%%%%%%%%%%%%%%%%%%%%%%%%

\section{Preliminary}

Throughout this paper, we will make the following assumptions.
\begin{Assumption}\label{ass0}
The functions $b:\mathbb{R} \mapsto \mathbb{R}$ and $\sigma : \mathbb{R} \mapsto \mathbb{R}$ are Lipschitz continuous.
\end{Assumption}

\begin{Assumption}\label{ass1}
%\begin{itemize}
%\item [(i)]
$(i)$ \ The function $h: \mathbb{R} \rightarrow \mathbb{R} $ is an increasing function and there exist $0 < m \leq M$ such that for each
$x$, $y \in \mathbb{R}$
\[ m|x-y| \leq |h(x)-h(y)| \leq  M|x-y|;\]

%\item [(ii)]
$(ii)$ \ The initial condition $\xi$ satisfies $h(\xi) \geq 0$.

%\end{itemize}
%$(iii)$ The mapping $h$ is a twice continuously differentiable function with bounded derivatives.
\end{Assumption}

For the purpose of the abstract analysis in this paper,
 we use the following notations and conclusions. Let $\mathcal{P}_1(\mathbb{R})$ be the set of probability measures on $\mathbb{R}$ with finite first moment. For $\mu$, $\nu \in \mathcal{P}_1(\mathbb{R}) $, the Wasserstein-1 distance between $\mu$ and $\nu$ is defined as:
\begin{eqnarray} \label{eq Wa dis}
W_1(\mu,\nu)=\sup_{g \in Lip_{1}(\mathbb{R})} \bigg|\int_\mathbb{R} g (\mathrm{d}\mu - \mathrm{d}\nu)\bigg|=\inf_{X \sim \mu ; \ Y \sim \nu}\mathbb{E}[|X-Y|],
\end{eqnarray}
where $Lip_{1}(\mathbb{R}):=\{g: \mathbb{R} \mapsto \mathbb{R}:  \sup_{x \neq y}\frac{|g(x)-g(y)|}{|x-y|}\leq 1\}$.

Define the function
\begin{equation}\label{II-eq2.2}
H: \mathbb{R} \times \mathcal P_1(\mathbb{R}) \ni(x,\nu) \mapsto H(x,\nu)=\int_\mathbb{R} h(x+z)\nu(\mathrm{d}z),
\end{equation}
and
\begin{equation}\label{II-eq2.3}
G_0: \mathcal P_1(\mathbb{R}) \ni \nu \mapsto \inf\{ x\geq 0: H(x,\nu)\geq 0 \}.
\end{equation}

 The following property of $G_0$ will be used several times
throughout this paper.
\begin{lemma}\label{lem1}(\cite[Lemma 2.2] {BDG})
Under Assumption \ref{ass1}, the function $G_0$ is Lipschitz continuous. Namely, for each $\mu$, $\nu \in \mathcal{P}_{1}(\mathbb{R})$
\begin{equation}\label{II-eq2.4}
|G_0(\mu)-G_0(\nu)| \leq \frac{M}{m}W_1(\mu, \nu).
\end{equation}
\end{lemma}

We now recall the existence and uniqueness result of \cite[Theorem 2.4] {BDG}.
\begin{definition}\label{defi1}
A couple of continuous processes $(X,K)=(X_t,K_t)_{t\in[0,T]}$ is said to be a flat deterministic solution to Eq.\eqref{II-eq1.1} if $(X, K)$ satisfy \eqref{II-eq1.1} with $X$ such that
$\mathbb{E}[\sup_{t\in [0, T]} |X_{t}|^{p}] < \infty$ for some $p \geq 2$ and $K$ being a nondecreasing deterministic function with $K_{0}=0$.
\end{definition}

Given this definition, we have the following result.
\begin{prp}\label{thm1}(\cite[Theorem 2.4] {BDG})
 Under Assumptions \ref{ass0} and \ref{ass1}, the mean reflected SDE \eqref{II-eq1.1} has a unique deterministic flat solution $(X,K)$.
 Moreover,
 \[ \forall t \geq 0, \ \ K_{t}=\sup_{s\leq t}\inf\{x \geq 0: \mathbb{E}[h(x+U_{s})] \geq 0\} = \sup_{s\leq t}G_0(\mu_s),\]
 where $(U_{t})_{0\leq t \leq T}$ is the process defined by:
\begin{equation}\label{II-eq2.7}
U_{t}= \xi + \int_0^t b(X_{s})\mathrm{d}s+\int_0^t \sigma(X_{s})\mathrm{d}B_{s},
\end{equation}
and $(\mu_{t})_{0\leq t \leq T}$ is the family of marginal laws of $(U_{t})_{0 \leq t \leq T}$.
\end{prp}

In the paper, $C_p$, $L_p$, etc. are  positive constants depending on some parameter $p$, and $C$, $L$, etc. are
constants depending on no specific parameter, whose value may be different from
line to line by convention.

\section{Small noise Large deviation principle}\label{Section 3}
%Fix $T \in (0,\infty)$ and $\xi \in \RR$.
For each $\epsilon > 0$, by  Proposition \ref{thm1}, there exists a unique flat deterministic solution $(X^\epsilon, K^\epsilon)$ to the following mean reflected SDE:
\begin{equation}\label{III-eq2.1}
\left\{
\begin{aligned}
& X_{t}^\epsilon=\xi + \int_0^t b(X_{s}^\epsilon)\mathrm{d}s+\ \sqrt{\epsilon}\int_0^t \sigma(X_{s}^\epsilon)\mathrm{d}B_{s}+\ K_{t}^\epsilon,\  t \in
[0,T], \\
& \mathbb{E}[h(X_{t}^\epsilon)]\geq 0, \quad
\int_0^t \mathbb{E}[h(X_{s}^\epsilon)]\mathrm{d}K_{s}^\epsilon=0, \ t \in [0,T].\\
\end{aligned}
\right.
\end{equation}
In this section, we consider a small noise LDP to $X^\epsilon$ as $\epsilon$ tending to 0.

 To state the main result in this section, we first introduce the following mean reflected ODE:
\begin{align}\label{III-eq2.2}
\left\{
\begin{aligned}
& X^0_{t}=\xi+\int_0^t b(X^0_s)\mathrm{d}s + K^0_t,  \  t \in [0,T], \\
& h(X^0_t)\geq 0, \quad
\int_0^t h(X^0_s)\mathrm{d}K^0_s=0, \  t \in [0,T].\\
\end{aligned}
\right.
\end{align}
 Proposition \ref{thm1} implies that there exists a unique solution $(X^0, K^0)$ to Eq. (\ref{III-eq2.2}).

For each $\varphi \in L^2([0,T],\mathbb{R})$, consider the so called skeleton equation:
\begin{equation}\label{III-eq2.16}
Y^\varphi_t= \xi + \int_0^t b(Y^\varphi_s)\mathrm{d}s + \int_0^t \sigma (Y^\varphi_s) \cdot \varphi(s)\mathrm{d}s + K^0_t,\ \ t\in [0,T].
\end{equation}
Here we stress that $K^0$ in Eq. \eqref{III-eq2.16} is the second part of the strong solution $(X^0, K^0)$ to Eq. \eqref{III-eq2.2}.
By standard arguments, we have the following result:
\begin{prp}\label{prp2}
Under Assumption \ref{ass0}, there exists a unique solution to Eq. \eqref{III-eq2.16}.
\end{prp}
%\begin{proof}
%The proof of this proposition is similar to the proof of Theorem 2.4 in \cite{BDG}, so we omit it here.
%By standard arguments, we know that  Eq.\eqref{III-eq2.16} have a unique strong.
%\end{proof}

We now state the main result in this section.
\begin{theorem}\label{thm3}
%For $\epsilon >0$, let $(X^\epsilon, K^\epsilon)$ be the unique flat deterministic solution to Eq.\eqref{III-eq2.1}.
Suppose that Assumptions \ref{ass0} and \ref{ass1} hold, then the family $\{X^\epsilon\}_{\epsilon>0}$ satisfies a LDP  on the space $C([0,T],\mathbb{R})$ as $\epsilon$ tend to 0 with the rate function
\begin{equation}\label{III-eq2.17}
 I(g):=\inf_{\{\varphi \in L^2([0,T],\mathbb{R}):g=Y^\varphi\}} \bigg\{ \frac{1}{2}\int_0^T|\varphi(s)|^2 \mathrm{d}s    \bigg\},\ \forall g\in C([0,T],\mathbb{R}),
\end{equation}
with the convention $\inf{\emptyset}=\infty$, here $Y^\varphi$ solves Eq. \eqref{III-eq2.16}.
\end{theorem}

\begin{proof}
According to Proposition \ref{prp2}, there exists a measurable mapping $\varGamma ^0:C([0,T], \mathbb{R})\rightarrow C([0,T], \mathbb{R})$ such that $Y^\varphi = \varGamma^0(\int_0^\cdot \varphi(s)\mathrm{d}s)$ for any $ \varphi\in L^2([0,T],\mathbb{R})$, here $Y^\varphi$ is the solution to Eq. \eqref{III-eq2.16}.

Let
\[S_N:=\{\varphi\in L^2([0,T],\mathbb{R}):\int_0^T |\varphi(s)|^2 \mathrm{d}s \leq N\},\]
and
\[\tilde{S}_N:=\{\phi: \phi\ is \  a \ \mathbb{R}\text{-}vauled\ \mathcal{F}_t\text{-}predictable\ process\ such\ that\ \phi(\omega)\in S_N,\
                 \mathbb{P}\text{-}a.s. \}\]
Throughout this paper, $S_N$ is endowed with the weak topology on $L^2([0,T],\mathbb{R})$ and it is a polish space.

For each $\epsilon>0$, consider the following SDE
\begin{eqnarray}\label{eq Z epsilon}
Z^\epsilon_t=\xi + \int_0^t b(Z_{s}^\epsilon) \mathrm{d}s+\ \sqrt{\epsilon}\int_0^t \sigma(Z_{s}^\epsilon) \mathrm{d}B_{s}+\ K_{t}^\epsilon.
\end{eqnarray}
Here we remark that $K^\epsilon$ in Eq. \eqref{eq Z epsilon} is the second part of the strong solution $(X^\epsilon, K^\epsilon)$ to Eq. \eqref{III-eq2.1} and $K^\epsilon$ is a nondecreasing deterministic function with $K^\epsilon_{0}=0$.

 Under Assumption \ref{ass0}, it is easy to see that there exists a unique strong solution $Z^\epsilon$ to Eq. \eqref{eq Z epsilon} and $Z^\epsilon=X^\epsilon$. By the
 Yamada-Watanabe theorem, there exists a measurable mapping $\varGamma^\epsilon: C([0,T], \mathbb{R})\rightarrow C([0,T],
 \mathbb{R}) $ (which dependent on $K^\epsilon$) such that
\[X^\epsilon=Z^\epsilon=\varGamma^{\epsilon}(B(\cdot)) ,\]
and applying the Girsanov theorem, for any $N>0$ and\ $\varphi^\epsilon \in \tilde{S}_N$,
\[Z^{\varphi^\epsilon}:=\varGamma^{\epsilon}(B(\cdot)+\frac{1}{ \sqrt{\epsilon} }\int_0^\cdot \varphi^\epsilon(s)\mathrm{d}s)\]
is the solution of the following SDE
\begin{equation}\label{III-eq2.18}
Z^{\varphi^\epsilon}_t=\xi + \int_0^t b(Z^{\varphi^\epsilon}_s)\mathrm{d}s
                       +\ \sqrt{\epsilon}\int_0^t\sigma(Z^{\varphi^\epsilon}_s)\mathrm{d}B_{s}
                       +\int_0^t\sigma(Z^{\varphi^\epsilon}_s)\varphi^\epsilon(s)\mathrm{d}s
                       +\ K_{t}^\epsilon .
\end{equation}
 We stress that $K^\epsilon$ in Eq. \eqref{III-eq2.18} is the second part of the strong solution $(X^\epsilon, K^\epsilon)$ to Eq. \eqref{III-eq2.1}; see Remark \ref{rem 1} for more details.

According to Theorem 3.2 in \cite{MSW} or Theorem 4.4 in \cite{WYJT}, Theorem \ref{thm3} is established once we have proved:

$(H_1)$\ for every $N < +\infty$ and any family $\{\varphi_n, n\in \mathbb{N}\} \subset S_N$ converging weakly to some elememt $\varphi$
as $n \rightarrow \infty $, $\varGamma^{0}(\int_0^ \cdot \varphi_n(s)\mathrm{d}s) $ converges to $\varGamma^{0}(\int_0^ \cdot \varphi(s)\mathrm{d}s)$
in the space $C([0,T], \mathbb{R})$, that is,
$$
\lim_{n\rightarrow\infty}\sup_{t\in[0,T]}\Big|\varGamma^{0}(\int_0^ \cdot \varphi_n(s)\mathrm{d}s)(t)-\varGamma^{0}(\int_0^ \cdot \varphi(s)\mathrm{d}s)(t)\Big|=0.
$$

$(H_2)$\ for every $N < +\infty $ and any family $\{\varphi^ \epsilon , \epsilon >0\}\subset \tilde{S}_N$ ,
\begin{eqnarray}\label{eq H2}
\lim_{\epsilon \rightarrow 0} \mathbb{E}[\sup_{0 \leq t \leq T}|Z^{\varphi^\epsilon}_t-Y^{\varphi^\epsilon}_t|^2] = 0,
\end{eqnarray}
where
$Y^{\varphi^{\epsilon}} =\varGamma^{0}(\int_0^\cdot \varphi^{\epsilon}(s)\mathrm{d}s)$.

The proofs of the above two claims are divided into the following two steps.

{\bf Step 1: Proof of Claim $(H_1)$.}

For simplicity we write $Y^n=\varGamma^{0}(\int_0^ \cdot \varphi_n(s)\mathrm{d}s) $ and $Y=\varGamma^{0}(\int_0^ \cdot \varphi(s)\mathrm{d}s)$.

Notice that $Y^n$ is the solution to (\ref{III-eq2.16}) with $\varphi$ replaced by $\varphi_n$, that is,
\begin{equation}\label{eq Yn}
Y^n_t= \xi + \int_0^t b(Y^n_s)\mathrm{d}s + \int_0^t \sigma (Y^n_s) \cdot \varphi_n(s)\mathrm{d}s + K^0_t,\ \ t\in [0,T].
\end{equation}
By Assumption 2.1, $\varphi_n\in S_N$, and the fact that $K^0$ is a nondecreasing deterministic continuous function with $K^0_{0}=0$, it is not difficulty to get that there exists a constant $C_{T,N,\xi,K^0_T}$, independent of $n$, such that
\begin{eqnarray}\label{eq Uniform Yn 01}
  \sup_{n\in\mathbb{N}}\sup_{t\in[0,T]}|Y^n_t|^2
\leq
   C_{T,N,\xi,K^0_T}.
\end{eqnarray}
and for any $0\le s< t\le T$,
\begin{eqnarray}\label{eq equicontinuous Yn}
  &&|Y^n_t-Y^n_s|^2\\
&\leq&
  C(t-s)\int_s^t |b(Y^n_l)|^2\mathrm{d}l + C\int_s^t |\sigma (Y^n_l)|^2\mathrm{d}l \int_s^t |\varphi_n(l)|^2dl + C|K^0_t-K^0_s|^2\nonumber\\
&\leq&
   C_{T,N,\xi,K^0_T}\Big((t-s)^2+(t-s)\Big)+ C|K^0_t-K^0_s|^2.\nonumber
\end{eqnarray}
According to the Arzela-Ascoli theorem, $\left\{Y^n\right\}_{n\in\mathbb{N}}$
is a precompact set in $C([0,T],\mathbb{R})$. It follows that there exists a subsequence, still denoted later by $\left\{Y^n\right\}_{n\in\mathbb{N}}$, and ${Z\in C([0,T],\mathbb{R})}$ such that
\begin{equation}\label{convergence 1}
		%\begin{split}	
	\lim_{n \rightarrow \infty}\sup_{0\le t\le T}\left|Y^n_t-Z_t\right|=0.
%\end{split}
	\end{equation}
Using Assumption 2.1, (\ref{convergence 1}) and the weak convergence of $\varphi_n$ to $\varphi$ in $L^2([0,T],\mathbb{R})$, passing to the limit $n\rightarrow\infty$ in (\ref{eq Yn}), we know that $Z$ is a solution to (\ref{III-eq2.16}). Due to the uniqueness of the solution for (\ref{III-eq2.16}), $Z=Y$.

The proof of Claim $(H_1)$ is complete.

{\bf Step 2: Proof of Claim $(H_2)$.}

We first establish a priori estimates.

Recall that $(X^\epsilon, K^\epsilon)$ and $(X^0, K^0)$ are the strong solutions to Eq. \eqref{III-eq2.1} and Eq. \eqref{III-eq2.2}, respectively. By using a similar arguments in the proof of Proposition 2.6 in \cite{BDG}, there exists a constant $L_{T,\xi}$ such that
\begin{eqnarray}\label{eq uniform X ep}
\sup_{\epsilon\in(0,1]}\mathbb{E}[\sup_{0\leq t \leq T}|X^\epsilon_t|^2]\leq L_{T,\xi}.
\end{eqnarray}
Since
\begin{eqnarray*}
X_{t}^\epsilon-\xi - \int_0^t b(X_{s}^\epsilon)\mathrm{d}s-\ \sqrt{\epsilon}\int_0^t \sigma(X_{s}^\epsilon)\mathrm{d}B_{s}=\ K_{t}^\epsilon,\  t \in
[0,T],
\end{eqnarray*}
by (\ref{eq uniform X ep}) and Assumption 2.1, it is easy to see that
\begin{eqnarray}\label{eq uniform K ep}
\sup_{\epsilon\in(0,1]}\sup_{t\in[0,T]}|K^\epsilon_t|\leq L_{T,\xi}.
\end{eqnarray}

Next we will prove that
\begin{eqnarray}\label{eq lim epsi 0}
\lim_{\epsilon\rightarrow0}\Big(\mathbb{E}\Big(\sup_{t\in[0,T]}|X^\epsilon_t-X^0_t|^2\Big)
                           +
                            \sup_{t\in[0,T]}|K^\epsilon_t-K^0_t|
                            \Big)
                            =0.
\end{eqnarray}

By Assumption \ref{ass0}, the Cauchy-Schwarz inequality, H\"{o}lder inequality and Burkholder-Davis-Gundy inequality, we get
\begin{equation}\label{XIII-eq1}
\begin{aligned}
   &\mathbb{E}[\sup_{t \in [0,T]}|X^{\epsilon}_t- X^{0}_t|^2] \\
&\leq 3\bigg\{ \mathbb{E}[\sup_{t\in [0,T]}|\int_0^t (b( X^{\epsilon}_s)-b(X^{0}_s))\mathrm{d}s |^2]
                          + \mathbb{E}[ \sup_{t\in [0,T]} |\sqrt{\epsilon}\int_0^t\sigma(X^{\epsilon}_s) \mathrm{d}B_s |^2 ]
                          + \sup_{t\in [0,T]}|K^{\epsilon}_t - K^{0}_t|^2   \bigg\}\\
&\leq C\bigg\{ T\mathbb{E}[ \int_0^T |b(X^{\epsilon}_s)-b(X^{0}_s)|^2 \mathrm{d}s ]
                          + \epsilon \mathbb{E}[ \int_0^T |\sigma(X^{\epsilon}_s)|^2 \mathrm{d}s ]
                          + \sup_{t\in [0,T]}|K^{\epsilon}_t - K^{0}_t|^2   \bigg\}\\
&\leq C_T\bigg\{ \int_0^T \mathbb{E}[\sup_{l \in [0,s]}|X^{\epsilon}_l - X^{0}_l|^2]\mathrm{d}s
                          + \epsilon(1+\mathbb{E}[\sup_{t \in [0,T]}| X^{\epsilon}_t|^2])
                          + \sup_{t \in [0,T]}|K^{\epsilon}_t - K^{0}_t|^2 \bigg\}\\
&\leq C_T\bigg\{ \int_0^T \mathbb{E}[\sup_{l \in [0,s]}|X^{\epsilon}_l - X^{0}_l|^2]\mathrm{d}s
                          + \epsilon(1 + L_{T,\xi})
                          + \sup_{t \in [0,T]}|K^{\epsilon}_t-\bar{K}^0_t|^2  \bigg\},
\end{aligned}
\end{equation}
here, we have used (\ref{eq uniform X ep}) in the last step.

Set
$$
U^\epsilon_t=\xi + \int_0^t b(X_{s}^\epsilon)\mathrm{d}s+\ \sqrt{\epsilon}\int_0^t \sigma(X_{s}^\epsilon)\mathrm{d}B_{s}\ \ \text{  and  } \ U^0_t=\xi + \int_0^t b(X_{s}^0)\mathrm{d}s.
$$
%and
%$$
%U^0(t)=\xi + \int_0^t b(X_{s}^0)\mathrm{d}s.
%$$
By Proposition \ref{thm1}, we know that
\[ \forall t \geq 0, \ \ K^\epsilon_{t}=\sup_{s\leq t}\inf\{x \geq 0: \mathbb{E}[h(x+U^\epsilon_{s})] \geq 0\} = \sup_{s\leq t}G_0(\mu^\epsilon_s),\]
and
\[ \forall t \geq 0, \ \ K^0_{t}=\sup_{s\leq t}\inf\{x \geq 0: \mathbb{E}[h(x+U^0_{s})] \geq 0\} = \sup_{s\leq t}G_0(\mu^0_s),\]
 where
$(\mu^\epsilon_{t})_{0\leq t \leq T}$ and $(\mu^0_{t})_{0\leq t \leq T}$ are the family of marginal laws of $(U^\epsilon_{t})_{0 \leq t \leq T}$ and $(U^0_{t})_{0 \leq t \leq T}$, respectively.

By  the above representations of $K^\epsilon$ and $K^0$, (\ref{eq Wa dis}), and Lemma \ref{lem1}, we have
\begin{eqnarray}\label{XIII-eq2}
\sup_{t \in [0,T]}|K^{\epsilon}_t - K^{0}_t|^2
&=&
\sup_{t \in [0,T]}|\sup_{s\leq t}G_0(\mu^\epsilon_s) - \sup_{s\leq t}G_0(\mu^0_s)|^2\nonumber\\
 &\leq&
 \sup_{t \in [0,T]}|G_0(\mu^\epsilon_t) - G_0(\mu^0_t)|^2\nonumber\\
  &\leq&(\frac{M}{m})^{2} \mathbb{E}[ \sup_{t\in [0,T]}|U^{\epsilon}_t - U^0_t|^2 ]\nonumber\\
  &\leq& C \bigg\{ \int_0^T \mathbb{E}[\sup_{l \in [0,s]}|X^{\epsilon}_l - X^{0}_l|^2]\mathrm{d}s
                                                                + \epsilon(1 + L_{T,\xi}) \bigg\}.
\end{eqnarray}
 By inserting (\ref{XIII-eq2}) into (\ref{XIII-eq1}) we obtain
\begin{equation}\label{XIII-eq3}
\mathbb{E}[\sup_{t \in [0,T]}|X^{\epsilon}_t- X^{0}_t|^2] \leq C\epsilon + C\int_0^T \mathbb{E}[\sup_{l \in [0,s]}|X^{\epsilon}_l - X^{0}_l|^2]\mathrm{d}s.
\end{equation}
Applying the Gronwall inequality and letting $\epsilon$ tend to 0,
\begin{equation*}\label{XIII-eq4}
\lim_{\epsilon \rightarrow 0} \mathbb{E}[\sup_{t \in [0,T]}|X^{\epsilon}_t- X^{0}_t|^2] = 0.
\end{equation*}
In combination with \eqref{XIII-eq2}, it follows that
\begin{equation*}\label{XIII-eq5}
\lim_{\epsilon \rightarrow 0} \sup_{t \in [0,T]}|K^{\epsilon}_t - K^{0}_t| =0.
\end{equation*}
The proof of (\ref{eq lim epsi 0}) is complete.

\vskip 0.2cm

Now, we will prove that there exists a constant $C_{T,N,\xi}$ such that
\begin{equation}\label{III-eq2.20}
\sup_{\epsilon\in(0,1]}\mathbb{E}[\sup_{t \in [0,T]}|Z^{\varphi^{\epsilon}}_t|^2] \leq C_{T,N,\xi}.
\end{equation}
By Assumption \ref{ass0}, the Cauchy-Schwarz inequality, H\"{o}lder inequality and Burkholder-Davis-Gundy inequality, and the facts that $\varphi^\epsilon \in \tilde{S}_N $ and  $K^\epsilon$ is a nondecreasing deterministic continuous function with $K^\epsilon_{0}=0$, we have
\begin{equation}\label{III-eq2.21}
\begin{aligned}
      &\mathbb{E}[\sup_{t \in [0,T]}|Z^{\varphi^{\epsilon}}_t|^2]\\
&\leq 5\bigg\{ |\xi|^2 + \mathbb{E}[\sup_{t \in [0,T]}|\int_0^t b(Z^{\varphi^\epsilon}_s) \mathrm{d}s|^2]
               + \mathbb{E}[\sup_{t \in [0,T]}|\sqrt{\epsilon}\int_0^t \sigma(Z^{\varphi^\epsilon}_s)\mathrm{d}B_{s}|^2]\\
&\ \ \ \ \ \ \ + \mathbb{E}[\sup_{t \in [0,T]}|\int_0^t \sigma(Z^{\varphi^\epsilon}_s)\varphi^{\epsilon}(s)\mathrm{d}s|^2]
               +|K^{\epsilon}_T|^2 \bigg\}\\
&\leq C\bigg\{ |\xi|^2+ T\mathbb{E}[\int_0^T|b(Z^{\varphi^\epsilon}_s)|^2\mathrm{d}s]
               + \epsilon\mathbb{E}[\int_0^T|\sigma(Z^{\varphi^\epsilon}_s)|^2\mathrm{d}s]\\
&\ \ \ \ \ \ \ + \mathbb{E}\big[(\int_0^T|\sigma(Z^{\varphi^\epsilon}_s)|^2\mathrm{d}s)(\int_0^T|\varphi^\epsilon(s)|^2\mathrm{d}s) \big]
               +|K^{\epsilon}_T|^2   \bigg\}\\
&\leq  C \bigg\{ |\xi|^2 + T\mathbb{E}[\int_0^T|b(Z^{\varphi^\epsilon}_s)|^2\mathrm{d}s] + (\epsilon+ N)\mathbb{E}[\int_0^T|\sigma(Z^{\varphi^\epsilon}_s)|^2\mathrm{d}s] + |K^\epsilon_T|^2\bigg\}\\
&\leq  C(1+|\xi|^2) +C\int_0^T\mathbb{E}[\sup_{l \in [0,s]}|Z^{\varphi^{\epsilon}}_l|^2]\mathrm{d}s,
\end{aligned}
\end{equation}
here, we have used (\ref{eq uniform K ep}) in the last step.
By the Gronwall inequality, we complete the proof of \eqref{III-eq2.20}.

\vskip 0.2cm

We are in
the position to prove Claim $(H_2)$, i.e., (\ref{eq H2}).

Recall that $Y^{\varphi^\epsilon}$ is the solution to (\ref{III-eq2.16}) with $\varphi$ replaced by ${\varphi^\epsilon}$.
Using arguments similar to that proving \eqref{XIII-eq3} and \eqref{III-eq2.21}, we get
\begin{equation}\label{III-eq2.22}
\begin{aligned}
       &\mathbb{E}[\sup_{t \in [0,T]}|Z^{\varphi^\epsilon}_t-Y^{\varphi^\epsilon}_t|^2]\\
&\leq 4\bigg\{ \mathbb{E}[\sup_{t \in [0,T]}|\int_0^t (b(Z^{\varphi^\epsilon}_s)-b(Y^{\varphi^\epsilon}_s))ds|^2]
     + \mathbb{E}[\sup_{t \in [0,T]}|\sqrt{\epsilon} \int_0^t \sigma(Z^{\varphi^\epsilon}_s)dB_s|^2]\\
&\ \ \ +\mathbb{E}[\sup_{t \in [0,T]}|\int_0^t (\sigma(Z^{\varphi^\epsilon}_s)-\sigma(Y^{\varphi^\epsilon}_s))\cdot \varphi^\epsilon (s)ds|^2]
    +\sup_{t \in [0,T]}|K^\epsilon_t-K^0_t|^2 \bigg\}\\
&\leq C\bigg\{T\mathbb{E}[\int_0^T|b(Z^{\varphi^\epsilon}_s)-b(Y^{\varphi^\epsilon}_s)|^2ds]+\epsilon\mathbb{E}[\int_0^T|\sigma(Z^{\varphi^\epsilon}_s)|^2ds]\\
&\ \ \  +N\mathbb{E}[\int_0^T|\sigma(Z^{\varphi^\epsilon}_s)-\sigma(Y^{\varphi^\epsilon}_s)|^2ds] +\sup_{t \in [0,T]}|K^\epsilon_t-K^0_t|^2 \bigg\}\\
&\leq C\bigg\{  \epsilon(1+\mathbb{E}[\sup_{t \in [0,T]}|Z^{\varphi^\epsilon}_t|^2])
              +\sup_{t \in [0,T]}|K^\epsilon_t-K^0_t|^2  \bigg\}\\
     &\ \ \ +C_{T,N}\int_0^T\mathbb{E}[\sup_{l \in [0,s]}|Z^{\varphi^\epsilon}_l-Y^{\varphi^\epsilon}_l|^2]\mathrm{d}s.\\
&\leq C_{T,N,\xi}\bigg\{ \epsilon+\sup_{t \in [0,T]}|K^\epsilon_t-K^0_t|^2 \bigg\}
     +C_{T,N}\int_0^T\mathbb{E}[\sup_{l \in [0,s]}|Z^{\varphi^\epsilon}_l-Y^{\varphi^\epsilon}_l|^2]ds.
\end{aligned}
\end{equation}
Here we have used \eqref{III-eq2.20} at the last step.

Hence, by the Gronwall inequality and \eqref{eq lim epsi 0}, letting $\epsilon \rightarrow 0 $, we obtain
\begin{equation}\label{III-eq2.23}
\begin{aligned}
\lim_{\epsilon \rightarrow 0}\mathbb{E}[\sup_{t \in [0,T]}|Z^{h^\epsilon}_t-Y^{h^\epsilon}_t|^2]=0.
%&\leq C\{\epsilon(1+\ee[\sup_{0\leq t \leq T}|Z^{h^\epsilon}_t|^2])+\sup_{0\leq t \leq T}|K^\epsilon_t-K^0_t|^2\}e^{C}\\
%&\stackrel{\epsilon\rightarrow 0}{\longrightarrow}0,
\end{aligned}
\end{equation}
The proof of Claim $(H_2)$ is finished, completing the whole proof of Theorem \ref{thm3}.

\end{proof}

\begin{remark}\label{rem 1}
  We stress that $K^\epsilon$ in Eq. \eqref{III-eq2.18} is the second part of the strong solution $(X^\epsilon, K^\epsilon)$ to Eq. \eqref{III-eq2.1}. This is somehow surprising.  The reason is that when perturbing the Brownian motion in the arguments of the mapping $\Gamma^{\epsilon}(\cdot)$, $K^\epsilon$ is already deterministic and hence it is not affected by the perturbation. For the details, we refer to Theorems 3.6, 3.8 and 4.4 in \cite{WYJT}.  An example is also introduced in \cite{WYJT};  see \cite[Example 1.1]{WYJT}.
\end{remark}

\section{Large deviations for short time}

Let $(X,K)$ be the unique flat deterministic solution to Eq. \eqref{II-eq1.1}. In this section, we consider  the small time asymptotic behavior of $X_t$ as $t \downarrow 0$.

We rewrite the equation \eqref{II-eq1.1} as, for any $\varepsilon\in(0,1]$,
\begin{align}
\left\{
\begin{aligned}\label{II-eq1.1-short time}
& X_{\epsilon t}=\xi + \epsilon\int_0^t b(X_{\epsilon s})\mathrm{d}s+\ \int_0^t \sigma(X_{\epsilon s})\mathrm{d}B_{\epsilon s}+\ K_{\epsilon t},\ t\in [0,1],  \\
& \mathbb{E}[h(X_{\epsilon t})]\geq 0, \quad
\int_0^t \mathbb{E}[h(X_{\epsilon s})]\mathrm{d}K_{\epsilon s}=0,\ t\in[0,1].\\
\end{aligned}
\right.
\end{align}
Set $V_t^\epsilon=\epsilon^{-\frac{1}{2}}B_{\epsilon t}$ and $\mathcal{G}^\epsilon_t=\mathcal{F}_{\epsilon t}$. We know that $\left\{V_t^\epsilon\right\}_{t\ge 0}$ is $\left\{\mathcal{G} ^\epsilon_t\right\}_{t\ge 0}$-Brownian motion. Then $X_{\epsilon t}$ has the same law of the solution for the following equation:
\begin{equation}\label{V-eq2.1}
\left\{
\begin{aligned}
&\tilde X_{t}^\epsilon=\xi + \epsilon \int_0^t b( \tilde X_{s}^{\epsilon} )\mathrm{d}s
                       + \ \sqrt{\epsilon}\int_0^t \sigma( \tilde X_{s}^{\epsilon} )\mathrm{d}B_{s}+\ \tilde K_{t}^\epsilon,\ t\in[0,1],  \\
&\mathbb{E}[h(\tilde X_{t}^\epsilon)]\geq0,\quad
\int_0^t \mathbb{E}[h(\tilde X_{s}^\epsilon)]\mathrm{d}\tilde K_{s}^{\epsilon}=0, \ t\in[0,1].\\
\end{aligned}
\right.
\end{equation}

Denote the solution of equation (\ref{V-eq2.1}) by $\tilde{X}^\epsilon=\{\tilde{X}^\epsilon_t\}_{t\in [0,1]}$.
To study the small time asymptotic behavior of the solution $X_t$ for equation \eqref{II-eq1.1} as $t \downarrow 0$, it equivalently discusses the asymptotic behavior of $X_{\epsilon t}$ as $\epsilon \downarrow 0$.
From the definition of
$\tilde{X}^\epsilon$, it turns to consider the asymptotic behavior of $\tilde{X}^\epsilon$ as $\epsilon\downarrow 0$, which is equivalent to Freidlin-Wentzell type LDP for $\tilde{X}^\epsilon$ as $\epsilon \downarrow 0$.

To state the main result in this section, we first introduce the following mean reflected ODE and the so called skeleton equation.
\begin{align}\label{V-eq2.2}
\left\{
\begin{aligned}
&  X^0_{t}=\xi + K^0_t, \ \ t\in[0,1],  \\
& h(X^0_{t})\geq 0, \quad
\int_0^t h(X^0_{s})\mathrm{d}K^0_s=0,\ t\in[0,1],\\
\end{aligned}
\right.
\end{align}
and for any $ \varphi \in L^2([0,1], \mathbb{R})$,
\begin{equation}\label{V-eq2.4}
\tilde Y^\varphi_t= \xi + \int_0^t \sigma(\tilde Y^\varphi_s)\varphi(s)\mathrm{d}s+K^0_t, \ \ t\in[0,1].
\end{equation}
By Proposition \ref{thm1}, we know that there exists a unique solution $(X^0,K^0)=(X^0_t,K^0_t)_{t\in[0,1]}$ to (\ref{V-eq2.2}), and
it is obvious that
\begin{eqnarray}\label{eq small time}
(X^0,K^0)=(\xi, 0).
\end{eqnarray}
We remark that $K^0$ in Eq. \eqref{V-eq2.4} is the second part of the strong solution $(X^0, K^0)$ to Eq. (\ref{V-eq2.2}). Hence, by (\ref{eq small time}),
We rewrite Eq. \eqref{V-eq2.4} as, for any $ \varphi \in L^2([0,1], \mathbb{R})$,
\begin{equation}\label{V-eq2.4 version 2}
\tilde Y^\varphi_t= \xi + \int_0^t \sigma(\tilde Y^\varphi_s)\varphi(s)\mathrm{d}s, \ \ t\in[0,1].
\end{equation}
Obviously, under Assumption \ref{ass0}, Eq. \eqref{V-eq2.4 version 2} has a unique solution $\tilde Y^\varphi=(\tilde Y^\varphi_t)_{t\in[0,1]}\in C([0,1],\mathbb{R})$.

We now state the main result in this section.
\begin{theorem}\label{thm4}
Suppose that Assumptions \ref{ass0} and \ref{ass1} hold.   Then the family
$\{\tilde{X}^\epsilon, \epsilon \in (0,1]\}$ satisfies LDP in the space $C([0,1],\mathbb{R})$ as $\epsilon$ tend to 0 with the rate function
\begin{equation}\label{V-eq2.5}
I(g):= \inf_{\{ \varphi \in L^2([0,1],\mathbb{R}): g=\tilde Y^\varphi\}}\{\frac{1}{2}\int_0^1|\varphi(s)|^2\mathrm{d}s \},\ \forall g\in C([0,1],\mathbb{R}),
\end{equation}
with the convention that $\inf{\emptyset}=\infty$, here $\tilde Y^\varphi$ solve Eq. \eqref{V-eq2.4 version 2}.

\end{theorem}

An interested reader would have little difficulties extending the
arguments in Section \ref{Section 3} to obtain Theorem \ref{thm4}. Here, we omit the proof.

\section{Malliavin calculus for MRSDE}

%Let $(\tilde \Omega, \tilde {\mathcal{F}}, \tilde{\mathbb{F}}, \tilde {\mathbb{P}})$ is a filtered probability space with
%$\tilde{\mathbb{F}}:= \{ \tilde{\mathcal{F}_t} \}_{t \in [0.1]}$ the augmented filtration generated by a one-dimensional standard Brownian  motion $\{\tilde{B_t}\}_{t \in [0,T]}$.

Suppose that $(X, K)$ is the unique flat deterministic solution to Eq. \eqref{II-eq1.1}.
In this section, we will calculate the Malliavin derivative of $X$.
%Here fix $T > 0$ and Let $(\Omega, \mathcal{F},\mathbb{P})$ is the canonical probability space associated with a Brownian motion $\{B_t\}_{t \in [0,T]}$.

We first present some preliminaries on Malliavin calculus from Chapter 1 in \cite{MCART}. Let $H:=L^2([0,T],\mathbb{R})$ be the Hilbert space
with the scalar inner product $\langle \cdot, \cdot \rangle_H$  and the norm $|\cdot|_H$. The corresponding Malliavin derivative of a Malliavin
differentiable random variable $F\in\mathcal{F}_T$ is denoted by $DF=\{D_t F\}_{t \in [0,T]}$. For each $p \geq 1$, the Sobolev space $D^{1,p}$ is defined as the
closure of the class of smooth random variables $F$ with respect to the norm
\[ \| F \|_{1,p}=[\mathbb{E}|F|^p+ \mathbb{E}|DF|_H^p]^{1/p}. \]
Denote $D^{1,\infty}:=\cap_{p\geq1}D^{1,p}$.

In this section, we require the following additional assumption.
\begin{Assumption}\label{ass2}
%The mapping $h$ is a twice continuously differentiable function with bounded derivatives.
The mappings $b$ and $\sigma$ are continuously differentiable functions, and their derivatives are denoted by $b^{\prime}$ and $\sigma^{\prime}$,
respectively.
\end{Assumption}

We state the first main result of this section.

\begin{theorem}\label{thm5}
%Let $(X, K)$ be the unique deterministic flat solution to Eq.\eqref{II-eq1.1}.
Suppose that Assumptions \ref{ass0}, \ref{ass1} and \ref{ass2} hold, then $X_t$ belongs to $D^{1,\infty}$ for any $t \in [0,T]$. Moreover, for each $ p \geq 1$,
\begin{equation}\label{IV-eq2.2}
\sup_{0\leq r \leq t}\mathbb{E}[\sup_{r\leq s \leq T}|D_r X_s|^p]< \infty ,
\end{equation}
and for $r \leq t \leq T$, the derivative $D_r X_t$ satisfies the following linear equation:
\begin{equation}\label{IV-eq2.3}
D_r X_t = \sigma (X_r) + \int_r^t \sigma^{\prime}(X_s)D_rX_s\mathrm{d}B_s +\int_r^t b^{\prime}(X_s)D_rX_s\mathrm{d}s,
\end{equation}
%for $r\leq t$ $a.e.$,
and $D_rX_t=0$, for $0 \leq t <r \leq T$.
\end{theorem}

\begin{proof}
Consider the Picard approximations given by
\begin{align}\label{IV-eq2.4}
&Y^0_t = \xi + K_t,  \nonumber\\
&Y^{n+1}_t=\xi + \int_0^t b(Y^n_s)\mathrm{d}s + \int_0^t \sigma(Y^n_s)\mathrm{d}B_s +K_t, \ \ n \geq 0. \nonumber\
\end{align}
Here $K$ is the second part of the strong solution $(X, K)$ to Eq. \eqref{II-eq1.1}.

By Assumption 2.1, it follows from a classical argument that, for any $p\geq1$,
\begin{equation}\label{IV-eq2.5}
\lim_{n\rightarrow\infty}\mathbb{E}[\sup_{0 \leq t \leq T}| Y^{n}_t-X_t |^p] = 0.
\end{equation}
Notice that $K$ is a nondecreasing deterministic continuous function with $K_{0}=0$. Using the same technique as the proof of Theorem 2.2.1 in \cite{MCART}, by induction on $n$, the following three statements hold:

(1) For each $n \geq 0$ and $t \in [0,T]$, $Y^n_t \in D^{1,\infty}$.

(2) For all $p > 1$
\begin{equation}\label{IV-eq2.6}
\psi_n(t):=\sup_{0\leq r \leq t}\mathbb{E}[\sup_{r \leq s \leq t}|D_r Y^n_s|^p] < \infty.
\end{equation}

(3) For some constants $c_1$ and $c_2$
\begin{equation}\label{IV-eq2.7}
\psi_{n+1}(t) \leq c_1+c_2\int_0^t \psi_n(s)\mathrm{d}s.
\end{equation}
%for some constants $c_1$ and $c_2$.\\
By \eqref{IV-eq2.6}, \eqref{IV-eq2.7} and Gronwall's lemma, $\{DY^n_t \}_{n\geq 0}$ are bounded in
$L^p(\Omega, H)$ uniformly in $n$ for all $p \geq 2$. Therefore, combining \eqref{IV-eq2.5} and Proposition 1.5.5 in \cite{MCART},
we obtain that $X_t \in D^{1,\infty}$. Finally, applying the operator $D$ to Eq. \eqref{II-eq1.1} and using Proposition 1.2.4 in \cite{MCART},
we get \eqref{IV-eq2.3}.

The proof of Theorem \ref{thm5} is complete.

\end{proof}

As an application of Theorem \ref{thm5}, we have
\begin{theorem}\label{thm6}
Suppose that Assumptions \ref{ass0}, \ref{ass1} and \ref{ass2} hold, and $\sigma$ is nondegenerate, i.e., $\sigma(x)\neq0$ for all $x\in\mathbb{R}$.  Then for any $0 \leq t \leq T$, the law of $X_t$ is absolutely continuous with respect to the Lebesgue measure on $\mathbb{R}$.
\end{theorem}

\begin{proof}
According to Corollary 2.1.2 in \cite{MCART}, it suffices to show that for $\mathbb{P}\text{-}a.s.$ $\omega \in \Omega$,
$\langle DX_t, DX_t \rangle_H >0 $.
By Theorem \ref{thm5}, for $r \leq t \leq T$, the derivative $D_r X_t$ satisfies the following linear equation:
\begin{equation*}\label{IV-eq2.8}
D_r X_t= \sigma (X_r)+ \int_r^t \sigma^{\prime}(X_s) D_rX_s\mathrm{d}B_s +\int_r^t b^{\prime}(X_s)D_rX_s\mathrm{d}s,
\end{equation*}
and $D_rX_t=0$, for $0 \leq t <r \leq T$.

Now we deduce a simpler expression for the derivative $DX_t$.
Consider the following processes:
\[Y_t = 1+ \int_0^t \sigma^{\prime}(X_s)Y_s \mathrm{d}B_s + \int_0^t b^{\prime}(X_s)Y_s \mathrm{d}s, \ \ 0 \leq t \leq T,\]
and
\[Z_t= 1- \int_0^t \sigma^{\prime}(X_s)Z_s \mathrm{d}B_s - \int_0^t [b^{\prime}(X_s)-\sigma^{\prime}(X_s)^{2}]Z_s \mathrm{d}s, \ \ 0 \leq t \leq T. \]
By the It\^{o} formula, it is easy to check that $Y_t Z_t=Z_t Y_t=1$, which implies that for any $0 \leq t \leq T$, $Y^{-1}_t=Z_t$.
Then considering the process $\{ Y_tY^{-1}_r\sigma(X_r), \ t \geq r \}$, we have
%One can check, it satisfies Eq.\eqref{MCF2}. Thus,
%It is easy to check that
\[D_rX_t=Y_tY^{-1}_r\sigma(X_r), \ \ r \leq t \leq T .\]
It follows that
\[\langle DX_t, DX_t \rangle_H=\int_0^t|Y_tY^{-1}_r\sigma(X_r)|^2\mathrm{d}r >0 .\]
We complete the proof.
\end{proof}

\vskip 0.4cm

\noindent{\bf  Acknowledgement.}\  This work is partly supported by the National Natural Science Foundation of China (No. 12131019, No. 11971456,  No. 11721101, No. 11871184) and the Fundamental Research Funds for the Central Universities (No. WK3470000024, No. WK0010000076).
%%%%%%%%%%%%%%%%%%%%%%%%%%%%%%%%%%%%%%%%%%%%%%%%%%%%%%%%%%%%%%%%%%%%%%%%%%%%%%%%%%%%%%%

\vskip 0.3cm

\noindent{\bf Disclosure statement} The authors have not disclosed any competing interests.

\end{document}